\documentclass[12pt]{article}
\usepackage{mathrsfs}
\usepackage{tipa}
\usepackage{amssymb}
\usepackage{amsfonts}
\usepackage{cite}
\usepackage{graphicx}
\usepackage{latexsym,amsmath,amssymb,amsfonts,amsthm}

\newtheorem{theorem}{Theorem}[section]
\newtheorem{lemma}[theorem]{Lemma}

\newtheorem{proposition}[theorem]{Proposition}

\begin{document}
\setcounter{page}{1}
\title{Faces of invariant convex sets in representations of nontrivial copolarity}
\author{Yi Shi$^{\ 1,2}$}
\date{}
\protect
\maketitle ~~~\\[-5mm]

\protect\footnotetext{\!\!\!\!\!\!\!\!\!\!\!\!\! {\bf 2020 Mathematics Subject Classification:}
57S15, 52A05.
}
\maketitle ~~~\\[-12mm]
{\footnotesize $^1$School of Mathematics and Statistics, Guizhou University of Finance and Economics, Guiyang 550025, China.\\[1mm]
$^2$School of Big Data Statistics, Guizhou University of Finance and Economics, Guiyang 550025, China.\\[1mm]
e-mail: shiyi\underline{\hbox to 0.2cm{}}math@163.com}\\[1mm]

{\bf Abstract:} \noindent Let $(V, G)$ be an orthogonal representation of a compact Lie group $G$ with nontrivial copolarity, and $\Sigma$ a fat section of $(V, G)$. If $E$ is a $G$-invariant compact convex set in $V$, then $P=E\cap\Sigma$ is a convex set in $\Sigma$. We prove that up to conjugacy the face structure of $E$ is completely determined by that of $P$ and that a face of $E$ is exposed if and only if the corresponding face of $P$ is exposed. Our result generalizes the result proved by Leonardo Biliotti, Alessandro Ghigi and Peter Heinzner in the case where $(V, G)$ is a polar representation.
\markright{\sl\hfill \hfill}\\

\section{Introduction}
\renewcommand{\thesection}{\arabic{section}}
\renewcommand{\theequation}{\thesection.\arabic{equation}}
\setcounter{equation}{0}

An orthogonal representation $(V, G)$ of a compact Lie group $G$ on $V$ is called \emph{polar} with \emph{section} $\Sigma$ if $\Sigma$ is a linear subspace of $V$ that intersects all $G$-orbits orthogonally (see \cite{Da}). In this situation, there is a finite group $W\subset O(\Sigma)$, called the Weyl group, such that the orbit spaces $V/G$ and $\Sigma/W$ are isometric. Let $\sigma: V\rightarrow\Sigma$ be the orthogonal projection onto $\Sigma$. By Kostant's Convexity Theorem (see e.g. \cite{BKM2,Ko,KS}), for any polar representation $(V, G)$, the bijection
\begin{eqnarray}
 \{G\text{-invariant subsets of } V\}&\longleftrightarrow& \{W\text{-invariant subsets of } \Sigma\} \nonumber \\
 E &\longrightarrow& P=E\cap\Sigma=\sigma(E), \\
  E= G \cdot P &\longleftarrow& P. \nonumber
\end{eqnarray}
maps convex subsets to convex subsets. Very recently, Bettiol, Kummer and Mendes \cite{BKM2} proved that the bijection $(1.1)$ also preserves some special classes of invariant convex subsets, such as convex semi-algebraic sets, spectrahedral shadows and rigidly convex sets. They also conjectured that $(1.1)$ preserves the class of spectrahedra. For further details about spectrahedra, spectrahedral shadows and rigidly convex sets, we refer the readers to \cite{BKM1,BKM2,BPT,KS,Ku1,Ku2,NP,Sc,SS}. For applications of $(1.1)$ in convex optimization problems, see Section 6 of \cite{BKM2} for details.

In \cite{GOT} Gorodski, Olmos and Tojeiro introduced a new invariant for isometric actions of compact Lie groups, the \emph{copolarity}. Roughly speaking, it measures how far from being polar the action is. A \emph{fat section} (term from \cite{Ma1,Ma2}, originally known as \emph{$k$-section} in \cite{GOT}) $\Sigma$ is defined for an orthogonal representation $(V, G)$ as follows:

(A) $\Sigma$ is a linear subspace of $V$ that intersects all $G$-orbits;

(B) for all $G$-regular $p\in \Sigma$ we have $\nu_{p}(G\cdot p)\subseteq T_{p}\Sigma$ with codimension $k$; and

(C) for all $G$-regular $p\in \Sigma$ and $g\in G$ such that $g\cdot p\in \Sigma$ we have $g\cdot\Sigma=\Sigma$.

\noindent The sections of a polar representation are $0$-sections in the above sense. we say that $(V, G)$ has \emph{nontrivial copolarity} if $V$ is not the unique fat section. we can associate to each fat section $\Sigma$ a Lie group $W := N_{G}(\Sigma)/Z_{G}(\Sigma)$ (the normalizer of $\Sigma$ in $G$ modulo its pointwise stabilizer), which acts on $\Sigma$ in a natural way. This $W$ is called the \emph{fat Weyl group} of $\Sigma$, although in general it can be almost any Lie group and hence is not a Weyl group in the classical sense. As in the case of polar representation, the orbit spaces $V/G$ and $\Sigma/W$ are isometric (see \cite{GOT,Ma1}). For details about copolarity and more examples of fat sections, we refer the readers to \cite{GG,GL,GOT,Ma1,Ma2,Me,PP}.

The recent work of Mendes \cite{Me} shows that the bijection $(1.1)$ also holds for orthogonal representation $(V, G)$ of nontrivial copolarity, in which case $\Sigma$ is a fat section and $W$ is the fat Weyl group of $\Sigma$ (see Proposition \ref{Me} for details). In fact, a bijection similar to $(1.1)$ has been shown to hold in a more general context in \cite{Me}. Namely, given two Euclidean vector spaces $V, V'$, two submetries $V\rightarrow X$ and $V'\rightarrow X'$, and an isometry $X\rightarrow X'$, the latter induces a bijection between saturated closed convex subsets of $V$ and $V'$, respectively. In the case where $(V, G)$ is an orthogonal representation of nontrivial copolarity, one get the bijection $(1.1)$ by taking $V\rightarrow X$ and $V'\rightarrow X'$ to be the natural quotient maps $V\rightarrow V/G$ and $\Sigma\rightarrow \Sigma/W$, and the isometry $V/G\rightarrow \Sigma/W$ to be the natural bijection induced by the inclusion $\Sigma\hookrightarrow V$. Mendes also introduced a more general class of fat sections, which are defined via submetry. For more about submetry, see \cite{BG,CG,KL,KR,LW,MR}.

Let $(V, G)$ be an orthogonal representation of nontrivial copolarity, and $(\Sigma, W)$ a fat section $\Sigma$ with fat Weyl group $W$, $\sigma: V\rightarrow\Sigma$ the orthogonal projection. If $E\subset V$ is a $G$-invariant compact convex set in $V$, then by Theorem C in \cite{Me} $P:=\sigma(E)=E\cap\Sigma$ is a $W$-invariant convex set in $\Sigma$. Denote by $\mathscr{F}(E)$ the faces of $E$ and by $\mathscr{F}(P)$ the faces of $P$. $G$ acts on $\mathscr{F}(E)$ and the Weyl group $W$ acts on $\mathscr{F}(P)$. If $Q$ is a face of $P$, then $F_{Q}:=\sigma^{-1}(Q)\cap E$ is a face of $E$ (see Proposition \ref{FQ}). Furthermore, we can prove the following theorem, which is a generalization of Theorem 0.1 in \cite{BGH3}:
\begin{theorem}\label{dl}
The map $\mathscr{F}(P)\rightarrow \mathscr{F}(E)$, $Q\mapsto F_{Q}$ induces a bijection between $\mathscr{F}(P)/W$ and $\mathscr{F}(E)/G$.
\end{theorem}

Theorem \ref{dl} was originally proved by Biliotti, Ghigi and Heinzner \cite{BGH3} in the case where $(V, G)$ is a polar representation.
In the case where $E$ is a polar orbitope, Theorem \ref{dl} was first established in \cite{BGH2}, and later reproved by Kobert and Scheiderer in \cite{KS} using a new approach. However, as pointed out in \cite{KS}, \cite{BGH2} proves a more precise result, implying in particular that the faces $F_{Q}$ of $P$ are themselves orbitopes under suitable groups. In fact, when $(V, G)$ is a polar representation, the authors proved that $F_{Q}=G^{Q^{\perp}}\cdot Q$ in both \cite{BGH2} and \cite{BGH3}, where $Q^{\perp}$ denote the orthogonal complement in $\Sigma$ of the affine hull of $Q$, and $G^{Q^{\perp}}:=\{g\in G\mid g\cdot x=x, \text{for all }x\in Q^{\perp}\}$. Our method of proving Theorem \ref{dl} is inspired by \cite{KS}. However, similar to \cite{KS}, we have not proved that $F_{Q}=G^{Q^{\perp}}\cdot Q$. Furthermore, we do not know whether $F_{Q}=G^{Q^{\perp}}\cdot Q$ still holds in our case.

An application of Theorem \ref{dl} is the following result, which is a generalization of Theorem 0.2 in \cite{BGH3}:
\begin{theorem}\label{d2}
The faces of E are exposed if and only if the faces of P are exposed.
\end{theorem}

For orthogonal representation $(V, G)$ with fat section $\Sigma$, if $x\in \Sigma$, then $W\cdot x=(G\cdot x)\cap\Sigma$ by Lemma 5.8 in \cite{GOT}. In Theorem \ref{dl}, if $E :=\textup{conv}(G\cdot x)$ is an orbitope in $V$, then by Theorem C in \cite{Me} $P=E\cap\Sigma=\textup{conv}((G\cdot x)\cap\Sigma)=\textup{conv}(W\cdot x)$ is an orbitope in $\Sigma$. Hence, by Theorem \ref{dl} and \ref{d2} we get the following theorem, which is a generalization of Theorem 1.1 in \cite{BGH2}:
\begin{theorem}\label{d3}
Let $(V, G)$ be an orthogonal representation of nontrivial copolarity, and $(\Sigma, W)$ a fat section $\Sigma$ with fat Weyl group $W$. If $x\in \Sigma$, then up to conjugacy the face structure of $E :=\textup{conv}(G\cdot x)$ is completely determined by that of $P :=\textup{conv}(W\cdot x)$ and that a face of $E$ is exposed if and only if the corresponding face of $P$ is exposed.
\end{theorem}

\section{Preliminaries}
\renewcommand{\thesection}{\arabic{section}}
\renewcommand{\theequation}{\thesection.\arabic{equation}}
\setcounter{equation}{0}
\setcounter{theorem}{0}

We start by recalling some basic definitions and results about convex sets. For more details see e.g. \cite{Sch}. Let $V$ be a real vector space with scalar product $\langle\cdot,\cdot\rangle$, and $E \subset V$ a compact convex subset of $V$. The \emph{relative interior} of $E$, denoted relint$E$, is the interior of E in its affine hull. By definition a \emph{face} of $E$ is a convex subset $F\subset E$ such that $x, y\in E$ and $\textup{relint}[x,y]\cap F\neq\emptyset$ implies $[x,y]\subset F$. A face distinct from $E$ and $\emptyset$ will be called a \emph{proper face}. Every nonzero vector $u$ defines an \emph{exposed face}
$$F_{u}(E):=\{x\in E\mid \langle x,u \rangle=\max_{y\in E}\langle y,u \rangle\}.$$
In general not all faces of a convex set are exposed, see Figure 1 in \cite{BGH3} for an example. Since $E$ is compact the faces are closed and hence compact. If $F_{1}$ and $F_{2}$ are faces of $E$ and they are distinct, then $\textup{relint}F_{1}\cap \textup{relint}F_{2}=\emptyset$. Moreover $E$ is the disjoint union of the relative interiors of its faces (see \cite{Sch}, P.74).

\begin{lemma}\label{chain}(Lemma 10 in \cite{BGH1}) If $E$ is a compact convex set and $F\subset E$ is a face, then there is a chain of faces $F=F_{n}\subsetneq F_{n-1}\subsetneq \cdots \subsetneq F_{0}=E$ which is maximal, in the sense that for any $i$ there is no face of $E$ strictly contained between $F_{i-1}$ and $F_{i}$.
\end{lemma}

For an orthogonal representation $(V, G)$ of a compact Lie group $G$, the \emph{orbitope} of a vector $x\in V$, denoted $\textup{conv}(G\cdot x)$, is the convex hull of the orbit $G\cdot x$ in $V$. See \cite{SSS} for further details about orbitopes.

The following Proposition is a special case of Theorem C in \cite{Me}:
\begin{proposition}\label{Me}(\cite{Me}) Let $(V, G)$ be an orthogonal representation, and $(\Sigma, W)$ a fat section $\Sigma$ with fat Weyl group $W$, $\sigma: V\rightarrow\Sigma$ the orthogonal projection onto $\Sigma$. Let $E\subset V$ be a $G$-invariant convex set in $V$, and $P:=\sigma(E)$. Then:

\noindent(a) $P=\sigma(E)=E\cap\Sigma$ and $E=G\cdot P$.

\noindent(b) For $x\in \Sigma$, $\textup{conv}(G\cdot x)\cap\Sigma=\sigma(\textup{conv}(G\cdot x))=\textup{conv}((G\cdot x)\cap\Sigma)=\textup{conv}(W\cdot x)$.

\noindent(c) The map $E\mapsto P=\sigma(E)=E\cap\Sigma$ is a bijection between $G$-invariant convex subsets of $V$ and $W$-invariant convex subsets of $\Sigma$.
\end{proposition}

\begin{proposition}\label{FQ}
Under the same assumptions and notations as in Proposition \ref{Me}, and further assume that $E$ is compact.

\noindent(a) Let $Q$ be a face of $P=\sigma(E)$, and $F_{Q}:=\sigma^{-1}(Q)\cap E$. Then $F_{Q}$ is a face of $E$, and $\sigma(F_{Q})=Q=F_{Q}\cap\Sigma$. Moreover, $F_{Q}$ is exposed if $Q$ is exposed.

\noindent(b) If there exist $0\neq u\in\Sigma$ and a face $F$ of $E$ such that $F:=F_{u}(E)$, then $Q:=\sigma(F)$ is an exposed face of $P$, and $F=F_{Q}:=\sigma^{-1}(Q)\cap E$.
\end{proposition}

\begin{proof} \noindent(a) For $x,y\in E$ such that $\textup{relint}[x,y]\cap F_{Q}\neq\emptyset$, $\sigma(\textup{relint}[x,y])\cap Q\neq\emptyset$, and hence $\sigma([x,y])\subset Q$ since $Q$ is a face. Thus $[x,y]\subset F_{Q}$, and hence $F_{Q}$ is a face of $E$. Since $\sigma(F_{Q})=Q\subset F_{Q}\cap \Sigma$ and $F_{Q}\cap \Sigma\subset \sigma(F_{Q})=Q$, we get $\sigma(F_{Q})=Q=F_{Q}\cap\Sigma$.
If there exists $0\neq u\in\Sigma$ such that $Q:=F_{u}(P)$ is an exposed face of $P=\sigma(E)$, then there exists $c\in \mathbb{R}$ such that $c=\langle x,u \rangle=\max_{y\in P}\langle y,u \rangle$ for any $x\in Q$. Since $\langle \sigma(z),u \rangle=\langle z,u \rangle$ for any $z\in E$, $c=\max_{y\in \sigma(E)}\langle y,u \rangle=\max_{y\in E}\langle y,u \rangle=\langle \bar{x},u \rangle$ for any $\bar{x}\in F_{Q}$. Thus $F_{Q}=F_{u}(E)$ is an exposed face of $E$.

\noindent(b) By the same argument as in (a), we can prove $Q=F_{u}(P)$ is an exposed face of $P$. Thus $F_{Q}=F_{u}(E)=F$.
\end{proof}

Note that in Proposition \ref{Me} and \ref{FQ}, $\Sigma$ can be any fat section, even $V$. In fact, if $\Sigma=V$, then $W=G$, hence in this trivial case these Propositions still hold. Therefore, we do not emphasize that $(V, G)$ has nontrivial copolarity in these two Propositions.

\section{Proof of Theorems}
\renewcommand{\thesection}{\arabic{section}}
\renewcommand{\theequation}{\thesection.\arabic{equation}}
\setcounter{equation}{0}
\setcounter{theorem}{0}

\begin{proposition}\label{F1} Under the same assumptions and notations as in Proposition \ref{Me}, and further assume that $E$ is compact. Let $0\neq u_{1}\in \Sigma$, $F_{1}:=F_{u_{1}}(E)$ an exposed face of $E$ and $Q_{1}:=\sigma(F_{1})$. Let $V_{1}:=u_{1}+ \nu_{u_{1}}(G\cdot u_{1})$, $\Sigma_{1}:=V_{1}\cap\Sigma$ and $\sigma_{1}: V_{1}\rightarrow\Sigma_{1}$ the orthogonal projection onto $\Sigma_{1}$. Set $G_{1}:=G_{u_{1}}=\{g\in G\mid g\cdot u_{1}=u_{1}\}$. Then:

\noindent(a) $F_{1}$ is a $G_{1}$-invariant compact convex set in $V_{1}$, and $\Sigma_{1}$ is a fat section of the orthogonal representation $(V_{1}, G_{1})$.

\noindent(b) $\sigma_{1}=\sigma\mid_{V_{1}}$, i.e., $\sigma_{1}(x)=\sigma(x)$ for any $x\in V_{1}$.

\noindent(c) $Q_{1}=\sigma_{1}(F_{1})=F_{1}\cap\Sigma_{1}=F_{u_{1}}(P)$ is an exposed face of $P$ and $F_{1}=G_{1}\cdot Q_{1}=F_{Q_{1}}:=\sigma^{-1}(Q_{1})\cap E$.

\end{proposition}

\begin{proof} \noindent(a) Let $\varphi: V\rightarrow \mathbb{R}$ be the function $\varphi(x)=\langle x,u_{1} \rangle$. Since $F_{1}=F_{u_{1}}(E)$ and $E$ is $G$-invariant, $\langle p,u_{1} \rangle=\max_{y\in E}\langle y,u_{1} \rangle=\max_{y\in G\cdot p}\langle y,u_{1} \rangle$ for any $p\in F_{1}$. Thus $0=X(\varphi(x))=\langle X,u_{1} \rangle$ for any $X\in T_{p}(G\cdot p)$, and hence $u_{1}\perp T_{p}(G\cdot p)$. Since the orbit $G\cdot p$ lies entirely in the round sphere in $V$ that is centered at $0$ and contains $p$, $p\perp T_{p}(G\cdot p)$, and hence $(u_{1}-p)\perp T_{p}(G\cdot p)$. Let $\gamma$ be the line segment from $u_{1}$ to $p$, then the geodesic $\gamma$ orthogonal to $G\cdot p$. Since a geodesic orthogonal to an orbit must be orthogonal to every orbit it meets, $\gamma$ orthogonal to $G\cdot u_{1}$, and hence $(u_{1}-p)\perp T_{u_{1}}(G\cdot u_{1})$. Thus $p\in V_{1}=u_{1}+ \nu_{u_{1}}(G\cdot u_{1})$, and hence $F_{1}\subset V_{1}$. Since $u_{1}\perp T_{u_{1}}(G\cdot u_{1})$, $0\in V_{1}$, so $V_{1}$ is a linear subspace of $V$. Clearly, $F_{1}=F_{u_{1}}(E)$ is a $G_{1}$-invariant compact convex set in $V_{1}$. By P.1598 in \cite{GOT} $\Sigma_{1}$ is a fat section of the orthogonal representation $(V_{1}, G_{1})$.

\noindent(b) By Lemma 5.10 in \cite{GOT} (see also Proposition 3.5 in \cite{Ma1}), $V_{1}=(V_{1}\cap\Sigma)\oplus (V_{1}\cap\Sigma^{\perp})$, where $\Sigma\oplus\Sigma^{\perp}=V$. Thus $\sigma_{1}=\sigma\mid_{V_{1}}$.

\noindent(c) By (a), (b) and Proposition \ref{Me}, $Q_{1}=\sigma_{1}(F_{1})=F_{1}\cap\Sigma_{1}$ and $F_{1}=G_{1}\cdot Q_{1}$. By the proof of Proposition \ref{FQ}, $Q_{1}=\sigma(F_{u_{1}}(E))=F_{u_{1}}(P)$, and hence $F_{Q_{1}}=F_{u_{1}}(E)=F_{1}$.
\end{proof}

We are now ready to prove Theorem \ref{dl}. The following arguments of proving Theorem \ref{dl} is inspired by the approach of Theorem 5.2 in \cite{KS}. Here we use the same assumptions and notions as in Proposition \ref{FQ}. For any $w\in W$ there exists $g\in N_{G}(\Sigma)$ with $w = gZ_{G}(\Sigma)$. It is easy to see that the projection $\sigma$ commute with the action of $N_{G}(\Sigma)$, and hence $F_{w\cdot Q} = g\cdot F_{Q}$. Thus the assignment $Q\rightarrow F_{Q}$ induces a map from $W$-orbits of faces of $P$ to $G$-orbits of faces of $E$. The following theorem asserts that this map is bijective, thus proving Theorem \ref{dl}:

\begin{theorem}\label{KS} Under the same assumptions and notations as in Proposition \ref{FQ}. Let $\tilde{F}$ be a face of the $E$.

\noindent(a) There exist a face $Q$ of $P$ and an element $g\in G$ such that $\tilde{F} = g\cdot F_{Q}$.

\noindent(b) If $Q'$ is another face of $P$ with $\tilde{F}\subseteq g'\cdot F_{Q'}$ for some $g'\in G$, then there exists $w\in W$ such that $w\cdot Q \subseteq Q'$.

\noindent In particular, $Q\rightarrow F_{Q}$ induces a bijective correspondence between $W$-orbits of faces of $P$ and $G$-orbits of faces of $E$, compatible with inclusion of faces.
\end{theorem}

\begin{proof} \noindent(a) By Lemma \ref{chain} there is a chain of faces $\tilde{F}=\tilde{F}_{n}\subsetneq \cdots \subsetneq \tilde{F}_{1}\subsetneq \tilde{F}_{0}=E$ which is maximal. Since any face is contained in an exposed face and $\tilde{F}_{i}$ is a maximal proper face of $\tilde{F}_{i-1}$, $\tilde{F}_{i}$ is an exposed face of $\tilde{F}_{i-1}$. Thus $\tilde{F}_{1}=F_{\tilde{u}_{1}}(E)$ for some $\tilde{u}_{1}\in V$. Since $\Sigma$ is a fat section of $(V,G)$, there exist $g_{1}\in G$ and a face $F_{1}$ of $E$ such that $u_{1}:=g_{1}\cdot\tilde{u}_{1}\in \Sigma$ and $F_{1}:=g_{1}\cdot\tilde{F}_{1}=F_{u_{1}}(E)$. With the same notations as in Proposition \ref{F1}, we get $F_{1}$ is a $G_{1}$-invariant compact convex set in $V_{1}$, and $\Sigma_{1}$ is a fat section of the orthogonal representation $(V_{1}, G_{1})$. Clearly, the chain of faces $g_{1}\cdot\tilde{F}\subsetneq \cdots \subsetneq g_{1}\cdot\tilde{F}_{2}\subsetneq g_{1}\cdot\tilde{F}_{1}=F_{1}\subsetneq g_{1}\cdot\tilde{F}_{0}=E$ is also maximal, and thus $g_{1}\cdot\tilde{F}_{2}$ is an exposed face of $F_{1}$. By the same argument as above, there exist $g_{2}\in G_{1}$, $u_{2}\in \Sigma_{1}$ and an exposed face $F_{2}:=(g_{2}g_{1})\cdot\tilde{F}_{1}=F_{u_{2}}(F_{1})$ of $g_{2}\cdot F_{1}=F_{1}$. Therefore by induction there exist a maximal chain of faces $F=F_{n}\subsetneq \cdots \subsetneq F_{1}\subsetneq F_{0}=E$, a chain of orthogonal representations with fat sections $(V_{n},G_{n},\Sigma_{n})\subset \cdots \subset(V_{1},G_{1},\Sigma_{1})\subset(V_{0},G_{0},\Sigma_{0})=(V,G,\Sigma)$, $g_{i}\in G_{i-1}$ and $u_{i}\in \Sigma_{i-1}$, such that $G_{i}:=(G_{i-1})_{u_{i}}\subset G_{i-1}$, $V_{i}:=u_{i}+\nu_{u_{i}}(G_{i-1}\cdot u_{i})\subset V_{i-1}$, $\Sigma_{i}=V_{i}\cap \Sigma_{i-1}$ is a fat section of $(V_{i},G_{i})$, and $F_{i}=(g_{i}g_{i-1}\cdots g_{1})\tilde{F}_{i}=F_{u_{i}}(F_{i-1})$ is a $G_{i}$-invariant compact convex set in $V_{i}$, for $1\leq i\leq n$.

Let $\sigma_{i}:V_{i}\rightarrow \Sigma_{i}$ be the orthogonal projection onto $\Sigma_{i}$, $Q_{i}:=\sigma_{i}(F_{i})$ and $Q:=Q_{n}$. By Proposition \ref{F1} $\sigma_{i}=\sigma_{i-1}\mid_{V_{i}}=\sigma_{i-2}\mid_{V_{i}}=\cdots =\sigma\mid_{V_{i}}$, $F_{i}=G_{i}\cdot Q_{i}$ and $Q_{i}=F_{u_{i}}(Q_{i-1})$. Since $F_{i}=F_{u_{i}}(F_{i-1})\subsetneq F_{i-1}$, $Q_{i}=F_{u_{i}}(Q_{i-1})\subsetneq Q_{i-1}$. Thus $Q=Q_{n}\subsetneq Q_{n-1}\subsetneq\cdots \subsetneq Q_{1}\subsetneq P$ is a chain of proper exposed faces. We will prove $F_{i}=F_{Q_{i}}:=\sigma^{-1}(Q_{i})\cap E$ by induction on $i$. For $i=1$, $F_{1}=F_{Q_{1}}$ by Proposition \ref{F1}. Let $i\geq1$ and assume that $F_{k}=F_{Q_{k}}$ for all $1\leq k\leq i$. Since $Q_{i+1}$ is a face of $P$, by Proposition \ref{FQ} $F_{Q_{i+1}}$ is a face of $E$. Since $F_{Q_{i+1}}\subsetneq F_{Q_{i}}=F_{i}$, $F_{Q_{i+1}}$ is a proper face of $F_{i}$. Since $Q_{i+1}=\sigma_{i+1}(F_{i+1})=\sigma(F_{i+1})$, $F_{i+1}\subseteq F_{Q_{i+1}}$. But $F_{i+1}$ is a maximal proper face of $F_{i}$, we have $F_{i+1}=F_{Q_{i+1}}$. Hence $F=F_{n}=F_{Q_{n}}=F_{Q}$. Set $g:=(g_{n}g_{n-1}\cdots g_{1})^{-1}$. Then $\tilde{F}=g\cdot F=g\cdot F_{Q}$.

\noindent(b) First, we prove that for any face $Q$ of $P$ there exists a subgroup $K\subset G$ such that $F_{Q}=K\cdot Q$. By Lemma \ref{chain} there is a chain of faces $Q=Q_{n}\subsetneq \cdots \subsetneq Q_{1}\subsetneq Q_{0}=P$ which is maximal. Since any face is contained in an exposed face, $Q_{i}$ is an exposed face of $Q_{i-1}$. Thus $Q_{1}=F_{u_{1}}(P)$ for some $u_{1}\in \Sigma$. Let $F_{1}:=F_{Q_{1}}$ and use the same notations as in Proposition \ref{F1}. Then $F_{1}=F_{u_{1}}(E)=G_{1}\cdot Q_{1}$ is a $G_{1}$-invariant convex set in $V_{1}$, and $\Sigma_{1}$ is a fat section of the orthogonal representation $(V_{1}, G_{1})$. Since $Q_{2}$ is an exposed face of $Q_{1}=\sigma_{1}(F_{1})=F_{1}\cap\Sigma_{1}$, $Q_{2}=F_{u_{2}}(Q_{1})$ for some $u_{2}\in \Sigma_{1}$. Let $F_{2}:=F_{Q_{2}}$, then $F_{2}=F_{u_{2}}(F_{1})=G_{2}\cdot Q_{2}$ with $G_{2}=(G_{1})_{u_{2}}$. Now let $F_{i}:=F_{Q_{i}}$ for $1\leq i\leq n$, $K:=G_{n}$ and use the same notations as in the proof of (a). Then by induction we have $F_{Q}=F_{n}=G_{n}\cdot Q_{n}=K\cdot Q$.

To prove (b), it follows from (a) that it suffices to show: If $Q, Q'$ are faces of $P$, and if $g\cdot F_{Q}\subseteq F_{Q'}$ for some $g\in G$, then there exists $w\in W$ with $w\cdot Q \subseteq Q'$. If $y\in \textup{relint}Q$, then $g\cdot y\in g\cdot F_{Q}\subseteq F_{Q'}$ since $Q\subset F_{Q}$. By above there is a subgroup $K'\subset G$ such that $F_{Q'}=K'\cdot Q'$. So there exists $p'\in Q'$ such that $g\cdot y\in K'\cdot p'$. Thus $p'\in Q'\cap K'\cdot (g\cdot y)\subset \Sigma\cap(G\cdot y)=W\cdot y$. Hence there exists $w\in W$ with $w\cdot y=p' \in Q'$. Since $w\cdot Q$ and $Q'$ are faces of $P$ and $w\cdot y\in Q'\cap\textup{relint}(w\cdot Q)$, we have $w\cdot Q \subseteq Q'$.
\end{proof}

\noindent {\it Proof of Theorem \ref{d2}.} Assume that all faces of $P$ are exposed. By Theorem \ref{dl} any face $F\in \mathscr{F}(E)$ can be written as $F=g\cdot F_{Q}$ for some $g\in G$ and $Q\in \mathscr{F}(P)$. By Proposition \ref{FQ} $F_{Q}$ is exposed since $Q$ is exposed, and hence $F=g\cdot F_{Q}$ is also exposed. Conversely, suppose that all faces of $E$ are exposed. For any face $Q\in \mathscr{F}(P)$, $F_{Q}\in \mathscr{F}(E)$ is exposed. Thus $F_{Q}:=F_{\tilde{u}}(E)$ for some $\tilde{u}\in V$. Since $\Sigma$ is a fat section of $(V,G)$, there exists $g\in G$ such that $u:=g\cdot \tilde{u}\in \Sigma$. By Theorem \ref{dl} there exists $w\in W$ such that $F_{w\cdot Q}:=g\cdot F_{Q}=F_{u}(E)$.  By Proposition \ref{FQ} $w\cdot Q=F_{u}(P)$ is an exposed face, and hence $Q$ is also exposed.
$\hfill \square$\\

\noindent$\mathbf{Funding.}$ The author is supported by National Natural Science Foundation of China (No.12261013).

 \end{document}